\documentclass{cstr}

\usepackage{amsmath,amsfonts,amssymb,latexsym,bm}
\usepackage{amsthm}
\usepackage{mathrsfs}
\usepackage[pdftex]{graphicx}
\usepackage{subfig}
\usepackage{url}
\usepackage{cite}

\usepackage{booktabs}

\newtheorem{theorem}{Theorem}

\newtheorem{proposition}{Proposition}

\usepackage{algorithm, algorithmic}
\renewcommand{\algorithmicrequire}{{\bf Input:}}
\renewcommand{\algorithmicensure}{{\bf Output:}}

\title{
Complex moment-based method with nonlinear transformation for computing large and sparse interior singular triplets
}

\author[1,*]{Akira Imakura}
\author[1]{Tetsuya Sakurai}

\affil[1]{University of Tsukuba, Japan}

\email{imakura@cs.tsukuba.ac.jp}

\begin{document}
\maketitle
\thispagestyle{titlepage}

\begin{abstract}
This paper considers computing interior singular triplets corresponding to the singular values in some interval.
Based on the concept of the complex moment-based parallel eigensolvers, in this paper, we propose a novel complex moment-based method with a nonlinear transformation.
We also analyse the error bounds of the proposed method and provide some practical techniques.
Numerical experiments indicate that the proposed complex moment-based method with the nonlinear transformation computes accurate singular triplets for both exterior and interior problems within small computation time compared with existing methods.
\end{abstract}

\section{Introduction}
Given a rectangular matrix $A \in \mathbb{R}^{m \times n}$ $(m \geq n)$, let 
\begin{equation*}
  A = U \Sigma V^{\rm T} = \sum_{i=1}^n \sigma_i {\bm u}_i {\bm v}_i^{\rm T}
\end{equation*}
be a singular value decomposition of $A$, where $\sigma_i$ are singular values and ${\bm u}_i$ and ${\bm v}_i$ are the corresponding left and right singular vectors, respectively, and $U = [{\bm u}_1, {\bm u}_2, \dots, {\bm u}_n]$, $V = [{\bm v}_1, {\bm v}_2, \dots, {\bm v}_n]$ and $\Sigma = {\rm diag}(\sigma_1, \sigma_2, \dots, \sigma_n)$.
To compute partial singular triplets specifically corresponding to the larger part of singular values, there are several projection-type methods such as Golub-Kahan-Lanczos method \cite{golub1996matrix}, Jacobi-Davidson type method \cite{hochstenbach2001jacobi} and randomized SVD algorithm \cite{halko2011finding}.
\par
This paper considers computing interior singular triplets corresponding to the singular values in some interval,
\begin{equation}
  (\sigma_i, {\bm u}_i, {\bm v}_i), \quad
  \sigma_i \in \Omega := [a,b],
  \label{eq:psvd}
\end{equation}
where $0 \leq a < b$.
One of the simplest ideas to compute \eqref{eq:psvd} is to apply some eigensolver for solving the corresponding symmetric eigenvalue problems,
\begin{equation}
  A^{\rm T} A {\bm v} = \sigma^2 {\bm v}
  \quad \mbox{or} \quad
  AA^{\rm T}  {\bm u} = \sigma^2 {\bm u}, \quad
  \sigma \in \Omega = [a, b].
  \label{eq:sep}
\end{equation}
One possible choice for solving interior eigenvalue problem \eqref{eq:sep} is complex moment-based parallel eigensolvers first proposed in \cite{sakurai2003projection} that are one of the hottest parallel methods for solving interior eigenvalue problems.
However, as shown in Section~\ref{sec:experiment}, this simple strategy does not work well in some situation due to the numerical instability.
\par
Based on the concept of the complex moment-based eigensolvers, in this paper, we propose a novel complex moment-based method to compute interior singular triplets \eqref{eq:psvd}.
From the analysis of error bounds, we also show that the accuracy of the proposed method can be improved via a nonlinear problem although the target problem is a linear singular value problem.
Some practical techniques are also provided.
\par
The remainder of this paper is organized as follows.
In Section~\ref{sec:ss}, we briefly introduce the complex moment-based parallel eigensolvers.
In Section~\ref{sec:propose}, we propose a novel complex moment-based method for computing interior singular triplets and analyse its error bound.
Here, we also propose an improvement technique using a nonlinear transformation.
Numerical results are reported in Section~\ref{sec:experiment}.
Section~\ref{sec:conclusion} concludes the paper.
\par
Throughout the paper, the following notations are used.
We define the range space of the matrix $V = [{\bm v}_1, {\bm v}_2, \ldots, {\bm v}_L]$ by $\mathcal{R}(V) := {\rm span}\{ {\bm v}_1, {\bm v}_2,$ $ \dots, {\bm v}_L \}$.
We also use MATLAB notations.
\section{Complex moment-based parallel eigensolvers}
\label{sec:ss}
The proposed method in this paper is based on the concept of the complex moment-based parallel eigensolvers first proposed in \cite{sakurai2003projection} by Sakurai and Sugiura.
Therefore, here, we briefly introduce the basic concepts of the complex moment-based eigensolvers for solving interior generalized eigenvalue problems of the form:
\begin{equation*}
  A {\bm x}_i = \lambda_i B {\bm x}_i, \quad
  A, B \in \mathbb{C}^{n \times n}, \quad
  {\bm x}_i \in \mathbb{C}^n \setminus \{ {\bm 0} \}, \quad
  \lambda_i \in \Omega \subset \mathbb{C},
\end{equation*}
where $zB-A$ is non-singular on a boundary $\Gamma$ of the target region $\Omega$.
\par
The complex moment-based eigensolvers construct a special subspace using contour integral:
\begin{equation}
  \mathcal{S}_\Omega = \mathcal{R}(S), \quad S = [S_0, S_1, \dots, S_{M-1}], \quad
  S_k := \frac{1}{2 \pi {\rm i} } \oint_\Gamma z^k (zB-A)^{-1} BV_{\rm in} {\rm d}z,
  \label{eq:set_s}
\end{equation}
where $L, M \in \mathbb{N}_+$ are the input parameters and $V_{\rm in} \in \mathbb{C}^{n \times L}$ is an input matrix.
For this subspace $\mathcal{S}_\Omega$, we have the following theorem; see e.g., \cite{imakura2016relationships}.
\begin{theorem}
  The complex moment-based subspace $\mathcal{S}_\Omega$ is equivalent to an invariant subspace with respect to the eigenvectors corresponding to the eigenvalues in a given region $\Omega \subset \mathbb{C}$, that is,
  \begin{equation*}
    \mathcal{S}_\Omega = {\rm span}\{ {\bm x}_i | \lambda_i \in \Omega \},
  \end{equation*}
  if and only if ${\rm rank}(S) = d$, where $d$ is the number of target eigenvalues.
\end{theorem}
\par
Based on this theorem, complex moment-based eigensolvers are mathematically designed on projection methods \cite{imakura2016relationships}.
Practical algorithms are derived by approximating the contour integral \eqref{eq:set_s} using the numerical integration rule:
\begin{equation}
  \widehat{S}_k := \sum_{j=1}^N \omega_j z_j^k (z_j B - A)^{-1} BV,
  \label{eq:linear}
\end{equation}
where $z_j$ is a quadrature point and $\omega_j$ is its corresponding weight.
\par
The most time-consuming part of using complex moment-based eigensolvers involves solving linear systems \eqref{eq:linear} at each quadrature point.
Since these linear systems can be independently solved, the complex moment-based eigensolvers have a good scalability that was demonstrated in previous research \cite{kestyn2016pfeast,iwase2017efficient}.
\par
Thanks to the high parallel efficiency, complex moment-based eigensolvers have attracted considerable attention.
Currently, there are several methods including direct extensions of Sakurai and Sugiura's approach \cite{sakurai2007cirr,ikegami2010filter,ikegami2010contour,imakura2014block,imakura2016relationships,imakura2017block,imakura2017structure}, the FEAST eigensolver \cite{polizzi2009density} developed by Polizzi and its improvements \cite{tang2014feast,guttel2015zolotarev,kestyn2016pfeast,kestyn2016feast}.
High-performance parallel software based on the complex moment-based eigensolvers have been developed \cite{z-Pares,FEAST}.
Complex moment-based machine learning algorithms have also been developed \cite{imakura2019complex,yano2021efficient}.
For details of these methods, refer to the study by \cite{sakurai2019scalable} and the references therein.
\section{Complex moment-based method for computing interior singular triplets}
\label{sec:propose}
Herein, we propose a complex moment-based method for computing interior singular triplets \eqref{eq:psvd} and analyse its error bound.
Based on the analysis, we propose an improvement technique using a nonlinear transformation to improve accuracy of the proposed method.
Some practical techniques are also provided.
\subsection{Derivation of the proposed method}
Based on the concept of the complex moment-based parallel eigensolvers, now, we have the following theorem.
\begin{theorem}
  \label{thm:svd}
  Let $L, M \in \mathbb{N}_+$ be the input parameters and $V_{\rm in} \in \mathbb{R}^{n \times L}$ be an input matrix.
  We define $S \in \mathbb{R}^{n \times LM}$ and $S_k \in \mathbb{R}^{n \times L}$ as follows:
  \begin{equation}
    S := [ S_0, S_1, \dots, S_{M-1}], \quad
    S_k := \frac{1}{2 \pi {\rm i} } \oint_\Gamma z^k (zI-A^{\rm T}A)^{-1} V_{\rm in} {\rm d}z,
    \label{eq:ci}
  \end{equation}
  where $\Gamma$ is a positively oriented Jordan curve around $[a^2,b^2]$.
  Then, the subspaces $\mathcal{R}(AS)$ and $\mathcal{R}(S)$ are equivalent to subspaces with respect to the left and right singular vectors corresponding to the singular values in a given interval $\Omega = [a,b]$, i.e.,
%
  \begin{equation*}
    \mathcal{R}(AS) = {\rm span}\{ {\bm u}_i | \sigma_i \in \Omega = [a,b] \},
    \quad \mbox{and} \quad
    \mathcal{R}(S) = {\rm span}\{ {\bm v}_i | \sigma_i \in \Omega = [a,b] \},
  \end{equation*}
  if and only if ${\rm rank}(S)=t$, where $t$ is the number of target singular values.
\end{theorem}
\begin{proof}
  Using the singular value decomposition of $A$, $A = U \Sigma V^{\rm T}$, and Cauchy's integral formula, the matrix $S_k$ can be decomposed as
  \begin{equation*}
    S_k 
    = \frac{1}{2 \pi {\rm i} } \oint_\Gamma \sum_{i=1}^n \frac{z^k}{z - \sigma_i^2} {\bm v}_i {\bm v}_i^{\rm T} V_{\rm in} {\rm d}z 
    = \sum_{\sigma_i \in \Omega} (\sigma_i^2)^k {\bm v}_i {\bm v}_i^{\rm T} V_{\rm in},
  \end{equation*}
  that proves Theorem~\ref{thm:svd}.
\end{proof}
%
%
\par
This theorem denotes that the target singular triplets \eqref{eq:psvd} can be obtained by some projection method with $\mathcal{R}(AS)$ and/or $\mathcal{R}(S)$ constructed by contour integral \eqref{eq:ci}.
In practice, the contour integral \eqref{eq:ci} is approximated by a numerical integration rule such as the $N$-point trapezoidal rule, as follows:
\begin{equation*}
  \widehat{S} := [\widehat{S}_0, \widehat{S}_1, \dots, \widehat{S}_{M-1}], \quad
  \widehat{S}_k := \sum_{j=1}^{N}\omega_j z_j^k (z_j I - A^{\rm T}A)^{-1} V_{\rm in},
\end{equation*}
where $(z_j,\omega_j), j = 1, 2, \dots, N$, are the quadrature points and the corresponding weights, respectively.
Then, the approximate singular triplets are computed by a projection method.
Here, we consider using two-sided projection method with subspaces $\mathcal{R}(A\widehat{S})$ and $\mathcal{R}(\widehat{S})$.
Note that one can also consider one-sided projection method.
\par
Let $\widetilde{U}$ and $\widetilde{V}$ be the orthogonal matrices whose columns are orthonormal basis of $\mathcal{R}(A\widehat{S})$ and $\mathcal{R}(\widehat{S})$, respectively.
From the definition of the subspace $\mathcal{R}(A\widehat{S})$, the matrix $\widetilde{U}$ is obtained by a QR factorization of $A \widetilde{V}$,
\begin{equation}
  A \widetilde{V} = \widetilde{U} B.
  \label{eq:avub}
\end{equation}
In a two-sided projection method, singular triplets are approximated as
\begin{equation*}
  ({\sigma}_i, {\bm u}_i, {\bm v}_i) 
  \approx (\widehat{\sigma}_i, \widehat{\bm u}_i, \widehat{\bm v}_i)
  = (\phi_i, \widetilde{U} {\bm p}_i, \widetilde{V} {\bm q}_i), \quad
\end{equation*}
and set $P = [{\bm p}_1, {\bm p}_2, \dots, {\bm p}_{LM}], Q = [{\bm q}_1, {\bm q}_2, \dots, {\bm q}_{LM}]$ and ${\Phi} = {\rm diag}(\phi_1, \phi_2, \dots$, $\phi_{LM})$.
Based on the Galerkin condition, the residual $R = A - (\widetilde{U} P) \Phi (\widetilde{V} Q)^{\rm T}$ is orthogonalized to the subspaces $\mathcal{R}(A \widehat{S})$ and $\mathcal{R}( \widehat{S})$, that is $\widetilde{U}^{\rm T}R \widetilde{V} = O$.
From \eqref{eq:avub}, we have $\widetilde{U}^{\rm T}A \widetilde{V} = B$, then the target singular triplets can be approximated by using a singular value decomposition of the matrix $B$,
\begin{equation}
  B = P \Phi Q = \sum_{\sigma_i \in \Omega} \phi_i {\bm p}_i {\bm q}_i^{\rm T}.
  \label{eq:bsvd}
\end{equation}
\par
To improve the accuracy, we can use an iteration technique.
The basic concept is that the matrix $\widehat{S}_0^{(\ell-1)}$ is iteratively calculated, from the initial matrix $\widehat{S}_0^{(0)} = V_{\rm in}$ as follows:
\begin{equation}
  \widehat{S}^{(\nu)}_0 := \sum_{j=1}^{N} \omega_j (z_jI-A^{\rm T}A)^{-1} \widehat{S}_0^{(\nu-1)}, \quad
  \nu = 1, 2, \ldots, \ell-1.
  \label{eq:iter1}
\end{equation}
Then, $\widehat{S}^{(\ell)}$ is constructed from $\widehat{S}_0^{(\ell-1)}$ by
\begin{equation}
  \widehat{S}^{(\ell)} := [\widehat{S}_0^{(\ell)}, \widehat{S}_1^{(\ell)}, \ldots, \widehat{S}_{M-1}^{(\ell)} ], \quad
  \widehat{S}^{(\ell)}_k := \sum_{j=1}^{N} \omega_j z_j^k (z_jI-A^{\rm T}A)^{-1} \widehat{S}_0^{(\ell-1)}.
  \label{eq:iter2}
\end{equation}
Additionally, for improving the numerical stability, we use a low-rank approximation with a threshold $\delta$ based on the singular value decomposition of $\widehat{S}^{(\ell)}$:
%
\begin{equation}
  \widehat{S}^{(\ell)} = [U_{S1}, U_{S2}] \left[
    \begin{array}{ll}
      \Sigma_{S1} & O \\
      O & \Sigma_{S2}
    \end{array}
  \right] \left[
    \begin{array}{ll}
      W_{S1}^{\rm T} \\
      W_{S2}^{\rm T}
    \end{array}
  \right] \approx U_{S1} \Sigma_{S1} W_{S1}^{\rm T},
  \label{eq:low-rank}
\end{equation}
where $\Sigma_{S1}$ is a diagonal matrix whose diagonal entries are the larger part of the singular values, and the columns of $U_{S1}, W_{S1}$ are the corresponding singular vectors.
Then, $U_{S1}$ is used for projection method instead of $\widehat{S}$.
\par
Because of the symmetric property of $A^{\rm T}A$, if quadrature points and the corresponding weights are symmetric about the real axis,
\begin{equation*}
  (z_j,\omega_j) = (\overline{z}_{j+N/2},\overline{\omega}_{j+N/2}),
  \quad j=1, 2, \dots, N/2,
\end{equation*}
we can reduce the number of linear systems as
\begin{equation*}
  \widehat{S}_k = 2 \sum_{j=1}^{N/2} {\rm Re}\left( \omega_j z_j^k (z_j I - A^{\rm T}A)^{-1} V_{\rm in} \right).
\end{equation*}
\par
The practical algorithm of the proposed method is shown in Algorithm~\ref{alg:ss-svd}.
One of the most time-consuming part of the complex moment-based method involves solving linear systems at each quadrature point in \eqref{eq:iter1} and \eqref{eq:iter2}.
However, as these linear systems can be independently solved, the proposed method is expected to exhibit good scalability in the same manner as the complex moment-based parallel eigensolvers.
\begin{algorithm}[t]
  \caption{A block SS--SVD method}
  \label{alg:ss-svd}
  \begin{algorithmic}[1]
    \REQUIRE $L, M, N, \ell \in \mathbb{N}_+, \delta \in \mathbb{R}, V_{\rm in} \in \mathbb{R}^{n \times L}, (z_j, \omega_j)$ for $j = 1, 2, \dots, N/2$
    \ENSURE Approximate singular triplet $(\widehat{\sigma}_i, \widehat{\bm u}_i, \widehat{\bm v}_i)$ for $i = 1, 2, \dots, \widehat{t}$
    \STATE Compute $\widehat{S}_0^{(\nu)} = 2 \sum_{j=1}^{N/2} {\rm Re}\left( \omega_j (z_jI-A^{\rm T}A)^{-1}\widehat{S}_0^{(\nu-1)} \right)$ for $\nu = 1, 2, \ldots, \ell-1$
    \STATE Compute $\widehat{S}_k^{(\ell)} = 2 \sum_{j=1}^{N/2} {\rm Re}\left( \omega_j z_j^k(z_jI-A^{\rm T}A)^{-1}\widehat{S}_0^{(\ell-1)} \right)$ for $k = 0, 1, \ldots, M-1$, and set $\widehat{S}^{(\ell)} = [\widehat{S}_0^{(\ell)}, \widehat{S}_1^{(\ell)}, \dots, \widehat{S}_{M-1}^{(\ell)}]$
    \STATE Compute low-rank approx. of $\widehat{S}^{(\ell)}$ using the threshold $\delta$: \\ $\widehat{S}^{(\ell)}= [U_{S1}, U_{S2}] [\Sigma_{S1}, O; O, \Sigma_{S2}] [W_{S1}, W_{S2}]^{\rm T} \approx U_{S1} \Sigma_{S1} W_{S1}^{\rm T}$, and set $\widetilde{V} = U_{S1}$
    \STATE Compute QR factorization of $AU_{S1}$: $[\widetilde{U},B] = {\rm qr}(A \widetilde{V})$
    \STATE Compute singular triplets $(\phi_i, {\bm p}_i, {\bm q}_i)$ of the matrix $B$ and set $(\widehat{\sigma}_i, \widehat{\bm u}_i, \widehat{\bm v}_i) = (\phi_i, \widetilde{U} {\bm p}_i, \widetilde{V} {\bm q}_i)$
  \end{algorithmic}
\end{algorithm}
\subsection{Error analysis of the proposed method}
Assume that $(z_j,\omega_j)$ satisfy
\begin{equation}
  \sum_{j=1}^N \omega_j z_j^k \left\{ 
    \begin{array}{ll}
      = 0, & k = 0, 1, \dots, N-2 \\
      \neq 0, & k = -1
    \end{array}
  \right..
  \label{eq:cond}
\end{equation}
Here, we have the following proposition for $\widehat{S}$; see, e.g., \cite{imakura2016error}.
\begin{proposition}
  \label{prop:sk=aks}
  Let $(z_j, \omega_j)$ satisfy \eqref{eq:cond}, then we have the following relationship,
  \begin{equation*}
    \widehat{S}_k 
    = \sum_{j=1}^{N} \sum_{i=1}^n \frac{\omega_j z_j^k}{z_j - \sigma_i^2}  {\bm v}_i {\bm v}_i^{\rm T} V_{\rm in}
    = \sum_{i=1}^n \sigma_i^{2k} \left( \sum_{j=1}^{N} \frac{\omega_j}{z_j - \sigma_i^2} \right)  {\bm v}_i {\bm v}_i^{\rm T} V_{\rm in}
    = (A^{\rm T} A)^k \widehat{S}_{0},
  \end{equation*}
  for $k = 0, 1, \dots, N-1$.
\end{proposition}
Let $f(\sigma_i)$ be a filter function
\begin{equation*}
  f(\sigma_i) := \sum_{j=1}^{N} \frac{\omega_j}{z_j - \sigma_i^2},
\end{equation*}
commonly used in the analyses of some eigensolvers \cite{tang2014feast, imakura2016error, imakura2016relationships, schofield2012spectrum, guttel2015zolotarev}.
Then, the matrix $\widehat{S}_k^{(\ell)}$ can be rewritten as
\begin{align*}
  \widehat{S}_k^{(\ell)} 
  &= \sum_{i=1}^n \left( \sum_{j=1}^{N} \frac{\omega_j z_j^k}{z_j - \sigma_i^2} \right) \left( \sum_{j=1}^{N} \frac{\omega_j}{z_j - \sigma_i^2} \right)^{\ell-1} {\bm v}_i {\bm v}_i^{\rm T} V_{\rm in} \\
  &= \sum_{i=1}^n (\sigma_i^2)^k \left( \sum_{j=1}^{N} \frac{\omega_j}{z_j - \sigma_i^2} \right)^\ell {\bm v}_i {\bm v}_i^{\rm T} V_{\rm in} \\
  &= \sum_{i=1}^n \sigma_i^{2k} (f(\sigma_i))^{\ell} {\bm v}_i {\bm v}_i^{\rm T} V_{\rm in} \\
  &= V \Sigma^{2k} V^{\rm T} V (f(\Sigma))^\ell V^{\rm T} V_{\rm in},
\end{align*}
where $f(\Sigma) = {\rm diag}(f(\sigma_1), f(\sigma_2), \dots, f(\sigma_n))$ that privides
\begin{equation}
  S^{(\ell)} = F^\ell K
  \label{eq:S=FK}
\end{equation}
with
\begin{equation*}
  F = V f(\Sigma^2) V^{\rm T}, \quad K = [V_{\rm in}, (A^{\rm T}A) V_{\rm in}, \dots, (A^{\rm T}A)^{M-1} V_{\rm in}].
\end{equation*}
\par
Using \eqref{eq:S=FK}, we have the following theorem for the error bound of the proposed method in the same manner as an error analysis of the subspace iteration method, see, e.g., Lemma~6.2.1 of \cite{chatelin1993eigenvalues} and Theorem~5.2 of \cite{saad2011numerical}.
\begin{theorem}
  \label{thm:error}
  Let $(\sigma_i,{\bm u}_i, {\bm v}_i)$ be exact singular triplets of $A$.
  Assume that $f(\sigma_i)$ are ordered by decreasing magnitude $|f(\sigma_i)| \geq |f(\sigma_{i+1})|$.
  Define $\mathcal{P}_U^{(\ell)}$ and $\mathcal{P}_V^{(\ell)}$ as orthogonal projectors onto the subspaces $\mathcal{R}( A\widehat{S}^{(\ell)} )$ and $\mathcal{R}(\widehat{S}^{(\ell)})$, respectively.
  We also define $\mathcal{P}_{LM}$ as the spectral projector with an invariant subspace ${\rm span}\{ {\bm v}_1, {\bm v}_2, \ldots, {\bm v}_{LM} \}$.
  Assume that the matrix $\mathcal{P}_{LM} K$ is full rank.
  Then, for each right singular vector ${\bm v}_i, i = 1, 2, \ldots, LM$, there exists a unique vector ${\bm s}_i \in \mathcal{R}(K)$ such that $\mathcal{P}_{LM} {\bm s}_i = {\bm v}_i$.
  Here, we have
  \begin{equation*}
    \| (I - \mathcal{P}_U^{(\ell)} ) {\bm u}_i \|_2
    \leq \alpha_i \beta_i \left| \frac{f(\sigma_{LM+1})}{f(\sigma_i)} \right|^\ell, \quad
    i = 1, 2, \ldots, LM,
  \end{equation*}
  and
  \begin{equation*}
    \| (I - \mathcal{P}_V^{(\ell)} ) {\bm v}_i \|_2
    \leq \beta_i \left| \frac{f(\sigma_{LM+1})}{f(\sigma_i)} \right|^\ell, \quad
    i = 1, 2, \ldots, LM,
  \end{equation*}
  where $\alpha_i = \max_{j\geq LM+1} \sigma_j / \sigma_i$ and $\beta_i = \| {\bm v}_i - {\bm s}_i \|_2$.
\end{theorem}
\begin{proof}
  Since $\mathcal{P}_{LM}K$ is full rank, there exists a unique vector ${\bm s}_i \in \mathcal{R}(K)$ as
  \begin{equation*}
    {\bm v}_i
    = \sum_{j=1}^{LM} \alpha_j \mathcal{P}_{LM} {\bm k}_j
    = \mathcal{P}_{LM} \left( \sum_{j=1}^{LM} \alpha_j {\bm k}_j \right)
    =: \mathcal{P}_{LM} {\bm s}_i,
  \end{equation*}
  where $K = [{\bm k}_1, {\bm k}_2, \ldots, {\bm k}_{LM}]$.
  Then, using ${\bm w}_i = (I - \mathcal{P}_{LM}) {\bm s}_i$, we have
  \begin{equation}
    {\bm s}_i = \mathcal{P}_{LM} {\bm s}_i + (I - \mathcal{P}_{LM}) {\bm s}_i = {\bm v}_i + {\bm w}_i.
    \label{eq:svw}
  \end{equation}
  \par
  Let ${\bm y}_i = (1 /\sigma_i) A (1/f(\sigma_i))^\ell F^\ell {\bm s}_i \in \mathcal{R}(A\widehat{S}^{(\ell)})$.
  Then, multiplying the matrix $(1/\sigma_i) A (1/f(\sigma_i))^\ell F^\ell$ to \eqref{eq:svw} from the left-side and considering ${\bm u}_i = (1/\sigma_i) A (1/f(\sigma_i))^\ell F^\ell {\bm v}_i$ and ${\bm w}_i = (I - \mathcal{P}_{LM}) {\bm s}_i = \mathcal{P}_{LM}^\perp {\bm s}_i$, we have
  \begin{equation*}
    {\bm y}_i - {\bm u}_i 
    = \frac{1}{\sigma_i} A \left( \frac{1}{f(\sigma_i)} F \right)^\ell {\bm w}_i
    = \frac{1}{\sigma_i} A \mathcal{P}_{LM}^\perp \left( \frac{1}{f(\sigma_i)} F \mathcal{P}_{LM}^\perp \right)^\ell {\bm w}_i,
  \end{equation*}
  that provides
  \begin{equation*}
    \| {\bm y}_i - {\bm u}_i \|_2
    \leq \frac{1}{\sigma_i} \| A \mathcal{P}_{LM}^\perp \|_2 \left| \! \left| \left( \frac{1}{f(\sigma_i)} F \mathcal{P}_{LM}^\perp \right)^\ell  \right| \! \right|_2 \|{\bm w}_i \|_2.
  \end{equation*}
  Here, using the relationship
  \begin{equation*}
    \min_{ {\bm y}_i \in \mathcal{R}(A \widehat{S}^{(\ell)}) } \| {\bm y}_i - {\bm u}_i \|_2
    = \| (I - \mathcal{P}_U^{(\ell)}) {\bm u}_i \|_2,
  \end{equation*}
  we thus have
  \begin{equation*}
    \| (I - \mathcal{P}_U^{(\ell)}) {\bm u}_i \|_2 
    \leq \| {\bm y}_i - {\bm u}_i \|_2 
    \leq \alpha_i \beta_i \left| \frac{ f(\sigma_{LM+1}) }{ f(\sigma_i) } \right|^\ell.
  \end{equation*}
  \par
  In the same manner, letting ${\bm z}_i = (1/f(\sigma_i))^\ell F^\ell {\bm s}_i \in \mathcal{R}(\widehat{S}^{(\ell)})$ and considering ${\bm v}_i = (1/f(\sigma_i))^\ell F^\ell {\bm v}_i$, we have
  \begin{equation*}
    \| (I - \mathcal{P}_V^{(\ell)}) {\bm v}_i \|_2 
    \leq \| {\bm z}_i - {\bm v}_i \|_2 
    \leq \beta_i \left| \frac{ f(\sigma_{LM+1}) }{ f(\sigma_i) } \right|^\ell,
  \end{equation*}
  that proves Theorem~\ref{thm:error}.
\end{proof}
\par
Theorem~\ref{thm:error} indicates that the accuracy of the proposed method depends on the subspace dimension $LM$.
Given a sufficiently large subspace, i.e., 
\begin{equation*}
  |f(\sigma_{LM+1})/f(\sigma_i)|^\ell \approx 0,
\end{equation*}
the target singular triplets can be obtained accurately, even if some singular values exist outside but near the region.
\subsection{An improvement technique using a nonlinear transformation}
\label{sec:nonlinear}
%
Theorem~\ref{thm:error} also indicates that if there is a cluster of singular values outside but near the region, then, to obtain accurate singular triplets, we have to use huge $LM$ that takes into account the number of the clustered singular values even though these are not the target.
This becomes huge computational costs.
Such a situation happens in the case that the singular values are uniformly distributed on the logarithmic scale.
\par
To overcome this difficulty, in this paper, inspired by complex moment-based nonlinear eigensolvers \cite{asakura2009numerical, asakura2010numerical, yokota2013projection, imakura2018block, beyn2012integral, van2016nonlinear}, we consider introducing a nonlinear transformation, $z = g(t)$, with an analytic monotonic increasing function $g$.
Then, we reset
\begin{equation*}
  S_k = \frac{1}{2 \pi {\rm i} } \oint_{\Gamma_t} t^k (g(t) I-A^{\rm T}A)^{-1} V_{\rm in} {\rm d}t,
\end{equation*}
where $\Gamma_t$ is a Jordan curve around $[g^{-1}(a^2),g^{-1}(b^2)]$.
Note that this also holds Theorem~\ref{thm:svd}.
The matrix $\widehat{S}_k$ is approximated by using contour integral as
\begin{equation*}
  \widehat{S}_k = \sum_{j=1}^{N}\psi_j t_j^k (g(t_j) I - A^{\rm T}A)^{-1} V_{\rm in}, 
\end{equation*}
where $(t_j,\psi_j), j = 1, 2, \dots, N$ are the quadrature points and the corresponding weights, respectively.
\par
Although Proposition~\ref{prop:sk=aks} usually does not hold in the case using the nonlinear transform, since we have
\begin{align*}
  S_k 
  &= \frac{1}{2 \pi {\rm i} } \oint_\Gamma \sum_{i=1}^n \frac{t^k}{g(t) - \sigma_i^2} {\bm v}_i {\bm v}_i^{\rm T} V_{\rm in} {\rm d}t \\
  &= \sum_{\sigma_i \in \Omega} (g^{-1}(\sigma_i^2))^k {\bm v}_i {\bm v}_i^{\rm T} V_{\rm in} \\
  &= (g^{-1}(A^{\rm T} A))^k S_0,
\end{align*}
we expect
\begin{align}
  \widehat{S}_k 
  &= \sum_{j=1}^{N} \sum_{i=1}^n \frac{\psi_j t_j^k}{g(t_j) - \sigma_i^2}  {\bm v}_i {\bm v}_i^{\rm T} V_{\rm in} \nonumber \\
  &\approx \sum_{i=1}^n (g^{-1}(\sigma_i))^{2k} \left( \sum_{j=1}^{N} \frac{\psi_j}{g(t_j) - \sigma_i^2} \right)  {\bm v}_i {\bm v}_i^{\rm T} V_{\rm in} \nonumber \\
  &= (g^{-1}(A^{\rm T} A))^k \widehat{S}_{0}.
  \label{eq:sk=aks_nonlinear}
\end{align}
Therefore, letting $f_g(\sigma)$ be a filter function
\begin{equation*}
  f_g(\sigma_i) := \sum_{j=1}^{N} \frac{\psi_j}{g(t_j) - \sigma_i^2},
\end{equation*}
then $\widehat{S}_k^{(\ell)}$ can be rewritten as
\begin{equation*}
  \widehat{S}_k^{(\ell)} 
  \approx V (g^{-1}(\Sigma))^{2k} V^{\rm T} V (f_g(\Sigma))^\ell V^{\rm T} V_{\rm in},
\end{equation*}
that provides
\begin{equation*}
  \widehat{S}^{(\ell)} \approx F_g^\ell K_g
\end{equation*}
with
\begin{equation*}
  F_g = V f_g(\Sigma) V^{\rm T}, \quad
  K_g = [V_{\rm in}, g^{-1}(A^{\rm T}A) V_{\rm in}, \dots, (g^{-1}(A^{\rm T}A))^{M-1} V_{\rm in}].
\end{equation*}
Thus, we have approximately the same error bounds as Theorem~\ref{thm:error} with $f_g(\sigma)$ instead of $f(\sigma)$ as the filter function.
\par
To set the function $g$ as $| f_g(\sigma_{LM+1}) / f_g(\sigma_i) | \ll | f(\sigma_{LM+1}) / f(\sigma_i) |$, the obtained accuracy is expected to be improved.
For example, if the singular values are uniformly distributed on the logarithmic scale, $g(t) = \exp(t)$ is a good choice to achieve small $| f_g(\sigma_{LM+1}) / f_g(\sigma_i) |$.
The algorithm of the proposed method with a nonlinear transform is shown in Algorithm~\ref{alg:ss-svd-nonlinear}.

\begin{algorithm}[t]
  \caption{A block SS--SVD method with nonlinear transformation}
  \label{alg:ss-svd-nonlinear}
  \begin{algorithmic}[1]
    \REQUIRE $L, M, N, \ell \in \mathbb{N}_+, \delta \in \mathbb{R}, V_{\rm in} \in \mathbb{R}^{n \times L}, (t_j, \psi_j)$ for $j = 1, 2, \dots, N/2$
    \ENSURE Approximate singular triplet $(\widehat{\sigma}_i, \widehat{\bm u}_i, \widehat{\bm v}_i)$ for $i = 1, 2, \dots, \widehat{t}$
    \STATE Compute $\widehat{S}_0^{(\nu)} = 2 \sum_{j=1}^{N/2} {\rm Re}\left( \psi_j (g(t_j)I-A^{\rm T}A)^{-1}\widehat{S}_0^{(\nu-1)} \right)$ for $\nu = 1, 2, \ldots, \ell-1$
    \STATE Compute $\widehat{S}_k^{(\ell)} = 2 \sum_{j=1}^{N/2} {\rm Re}\left( \psi_j t_j^k(g(t_j)I-A^{\rm T}A)^{-1}\widehat{S}_0^{(\ell-1)} \right)$ \\ for $k = 0, 1, \ldots, M-1$, and set $\widehat{S}^{(\ell)} = [\widehat{S}_0^{(\ell)}, \widehat{S}_1^{(\ell)}, \dots, \widehat{S}_{M-1}^{(\ell)}]$
    \STATE Compute low-rank approx. of $\widehat{S}^{(\ell)}$ using the threshold $\delta$: \\ $\widehat{S}^{(\ell)}= [U_{S1}, U_{S2}] [\Sigma_{S1}, O; O, \Sigma_{S2}] [W_{S1}, W_{S2}]^{\rm T} \approx U_{S1} \Sigma_{S1} W_{S1}^{\rm T}$, and set $\widetilde{V} = U_{S1}$
    \STATE Compute QR factorization of $AU_{S1}$: $[\widetilde{U},B] = {\rm qr}(A \widetilde{V})$
    \STATE Compute singular triplets $(\phi_i, {\bm p}_i, {\bm q}_i)$ of the matrix $B$ and set $(\widehat{\sigma}_i, \widehat{\bm u}_i, \widehat{\bm v}_i) = (\phi_i, \widetilde{U} {\bm p}_i, \widetilde{V} {\bm q}_i)$
  \end{algorithmic}
\end{algorithm}
\subsection{Practical techniques}
Here, we introduce two practical techniques: an efficient residual norm computation and a spurious singular value detection.
\subsubsection{Efficient residual norm computation}
Computation of the residual norms for obtained singular triplets is computationally costly when the problem size is large.
Here, we consider an efficient computation of the residual 2-norm $\|{\bm r}_i\|_2 = \| A^{\rm T} \widehat{\bm u}_i - \widehat{\sigma}_i \widehat{\bm v}_i \|_2$.
\par
Since the proposed method based on the Galerkin-type two-sided projection provides $B = P \Phi Q^{\rm T}$ \eqref{eq:bsvd} and $A \widetilde{V} = \widetilde{U} B$ \eqref{eq:avub}, we have $A \widetilde{V} Q = \widetilde{U} P \Phi$, i.e., we have $A \widehat{\bm v}_i = \widehat\sigma_i \widehat{\bm u}_i$ for all $i$.
Therefore, the residual 2-norm $\| {\bm r}_i \|_2$ can be replaced as
\begin{equation*}
  \| {\bm r}_i \|_2 
  = \left| \! \left| \frac{1}{\sigma_i} A^{\rm T} A \widehat{\bm v}_i - \widehat{\sigma}_i \widehat{\bm v}_i \right| \! \right|_2
  = \frac{1}{\sigma_i} \| A^{\rm T}A \widehat{\bm v}_i - \widehat\sigma_i^2 \widehat{\bm v}_i \|_2.
\end{equation*}
\par
Let $\widehat{S}^{(\ell)}_+ = [\widehat{S}^{(\ell)}_1, \widehat{S}^{(\ell)}_2, \dots, \widehat{S}_M^{(\ell)}]$.
For the case of Algorithm~\ref{alg:ss-svd}, from Proposition~\ref{prop:sk=aks}, we have $\widehat{S}^{(\ell)}_+ = (A^{\rm T}A) \widehat{S}^{(\ell)}$.
Therefore, the residual 2-norm $\| {\bm r}_i \|$ is rewritten as
\begin{align}
  \| {\bm r}_i \|_2
  &= \frac{1}{\sigma_i} \| A^{\rm T}A U_{S1} {\bm q}_i - \widehat{\sigma}_i^2 U_{S1} {\bm q}_i \|_2 \nonumber \\
  &= \frac{1}{\sigma_i} \left\| A^{\rm T} A \widehat{S}^{(\ell)} V_{S1} \Sigma_{S1}^{-1} {\bm q}_i - \widehat{\sigma}_i^2 U_{S1} {\bm q}_i \right\|_2 \nonumber \\
  &= \frac{1}{\sigma_i} \left\| \widehat{S}_+^{(\ell)} V_{S1} \Sigma_{S1}^{-1} {\bm q}_i - \widehat{\sigma}_i^2 U_{S1} {\bm q}_i \right\|_2,
  \label{eq:residual}
\end{align}
that achieves an efficient residual 2-norm computation without a matrix product for $A$, since the matrix $\widehat{S}^{(\ell)}_+$ is efficiently obtained by contour integral as well as \eqref{eq:iter1} and \eqref{eq:iter2}.
\par
Next, we consider the case of using the nonlinear transformation.
Assuming that the relative residual of symmetric eigenvalue problem is almost invariant to nonlinear transform,
\begin{equation*}
  \frac{ \| A^{\rm T}A \widehat{\bm v}_i - \widehat{\sigma}_i^2 \widehat{\bm v}_i \|_2 }{ \| A^{\rm T}A \|_2 }
  \approx \frac{ \| g^{-1}(A^{\rm T}A) \widehat{\bm v}_i - g^{-1}(\widehat{\sigma}_i^2) \widehat{\bm v}_i \|_2 }{ \| g^{-1}(A^{\rm T}A) \|_2 },
\end{equation*}
we have
\begin{equation}
  \| {\bm r}_i \|_2 \approx \frac{1}{\sigma_i} \| g^{-1}(A^{\rm T}A) \widehat{\bm v}_i - g^{-1}(\widehat{\sigma}_i^2) \widehat{\bm v}_i \|_2 \frac{ \| A^{\rm T}A \|_2 }{ \| g^{-1}(A^{\rm T}A) \|_2 }.
  \label{eq:res}
\end{equation}
Here, from \eqref{eq:sk=aks_nonlinear}, we have $\widehat{S}_+^{(\ell)} \approx g^{-1}( A^{\rm T}A ) \widehat{S}^{(\ell)}$.
Then, the 1st part of \eqref{eq:res} can be approximated as
\begin{align*}
  &\frac{1}{\sigma_i} \| g^{-1}(A^{\rm T}A) \widehat{\bm v}_i - g^{-1}(\widehat{\sigma}_i^2) \widehat{\bm v}_i \|_2 \nonumber \\
  &\quad = \frac{1}{\sigma_i} \| g^{-1}(A^{\rm T}A) U_{S1} {\bm q}_i - g^{-1}(\widehat{\sigma}_i^2) U_{S1} {\bm q}_i \|_2 \nonumber \\
  &\quad = \frac{1}{\sigma_i} \left\| g^{-1}(A^{\rm T} A) \widehat{S}^{(\ell)} V_{S1} \Sigma_{S1}^{-1} {\bm q}_i - g^{-1}(\widehat{\sigma}_i^2) U_{S1} {\bm q}_i \right\|_2 \nonumber \\
  &\quad \approx \frac{1}{\sigma_i} \left\| \widehat{S}_+^{(\ell)} V_{S1} \Sigma_{S1}^{-1} {\bm q}_i - g^{-1}(\widehat{\sigma}_i^2) U_{S1} {\bm q}_i \right\|_2 \nonumber \\
  &=: \| \widetilde{\bm r}_i \|_2.
\end{align*}
Since the 2nd part of \eqref{eq:res} is constant for $i$, we have
\begin{equation*}
  \frac{ \| A^{\rm T}A \|_2 }{ \| g^{-1}(A^{\rm T}A) \|_2 }
  \approx \frac{ \| {\bm r}_{i'} \|_2 }{ \| \widetilde{\bm r}_{i'} \|_2 },  \quad
  i' \in \{1, 2, \dots, \widehat{t}\}.
\end{equation*}
As a result, we can efficiently approximate all $\|{\bm r}_i\|_2$ by 
\begin{equation}
  \| {\bm r}_i\|_2 \approx \mu \| \widetilde{\bm r}_i \|_2, \quad
  \mu = \frac{ \| {\bm r}_{i'} \|_2 }{ \| \widetilde{\bm r}_{i'} \|_2 },
  \label{eq:res_nonlinear}
\end{equation}
with only one exact residual 2-norm $\| {\bm r}_{i'} \|_2 = \| A^{\rm T} \widehat{\bm u}_{i'} - \widehat\sigma_{i'} \widehat{\bm v}_{i'} \|_2$.
\subsubsection{Spurious singular value detection}
Similarly to other methods, the proposed method may compute spurious singular values in the target region.
In the proposed method, both subspaces $\mathcal{R}(AU_{S1})$ and $\mathcal{R}(U_{S1})$ are constructed from the matrix $U_{S1}$ obtained by a low-rank approximation \eqref{eq:low-rank} of $\widehat{S}$.
Here, we focus on this low-rank approximation and introduce a technique for spurious singular value detection.
\par
Let $U_{S1}$ be split as $U_{S1} = [U_{S1}^{(1)}, U_{S1}^{(2)}]$, where $U_{S1}^{(1)}$ and $U_{S1}^{(2)}$ correspond to large and small singular values in $\Sigma_{S1}$, respectively.
Then, since $\widetilde{V} = U_{S1}$, the approximation of right singular vectors ${\bm v}_i = \widetilde{V}{\bm q}_i$ is replaced as
\begin{equation*}
  {\bm v}_i = U_{S1} {\bm q}_i = [U_{S1}^{(1)}, U_{S1}^{(2)}] 
  \left[
    \begin{array}{c}
      {\bm q}_i^{(1)} \\ {\bm q}_i^{(2)}
    \end{array}
  \right].
\end{equation*}
From Theorem~\ref{thm:svd}, we expect ${\rm span}\{ {\bm v}_i | \sigma_i \in \Omega \} \approx \mathcal{R}(U_{S1}^{(1)})$.
Therefore, the magnitude of the elements of ${\bm q}_i^{(1)}$ is expected to be much larger than that of ${\bm q}_i^{(2)}$.
\par
Based on this concept, we detect spurious singular values using the following index:
\begin{equation*}
  \tau_i 
  := \frac{ \| {\bm q}_i \|_2^2 }{ \| {\bm q}_i \|_{\Sigma_{S1}^{-1}}^2}
  = \frac{ {\bm q}_i^{\rm T}{\bm q}_i }{ {\bm q}_i^{\rm T} {\Sigma_{S1}^{-1}} {\bm q}_i}.
\end{equation*}
If $\tau_i$ is smaller than some constant $\varepsilon$, we treat the eigenpair as a spurious singular values.
\section{Numerical experiments}
\label{sec:experiment}
Here, we evaluate the performance of the proposed method (Algorithms~\ref{alg:ss-svd} and \ref{alg:ss-svd-nonlinear}).
For Algorithm~\ref{alg:ss-svd-nonlinear}, we set $g$ as $z = g(t) = \exp(t)$.
\par
Let the quadrature points $z_j$ and $t_j$ be on an ellipse with center $\gamma$, major axis $\rho$ and aspect ratio $\alpha$, i.e.,
\begin{equation*}
  z_j = t_j = \gamma + \rho \left(\cos(\theta_j) + \alpha {\rm i} \sin(\theta_j)\right), \quad
  \theta_j = \frac{2 \pi}{N} \left(j-\frac{1}{2}\right), \quad
  j = 1, 2, \ldots, N.
\end{equation*}
The corresponding weights are set as
\begin{equation*}
  \omega_j = \frac{\rho}{N} \left(\alpha\cos(\theta_j) + {\rm i} \sin(\theta_j)\right), \quad
  \psi_j = \frac{\rho}{N} \exp(t_j) \left(\alpha\cos(\theta_j) + {\rm i} \sin(\theta_j)\right).
\end{equation*}
We set $(\gamma,\rho,\alpha) = ((a^2 + b^2)/2, (b^2 - a^2) / 2, 0.1)$ for Algorithm~\ref{alg:ss-svd} and $(\gamma,\rho,\alpha) = (\log(a) + \log(b),  \log(b) - \log(a), 0.1)$ for Algorithm~\ref{alg:ss-svd-nonlinear}.
\par
In these numerical experiments, the algorithms were implemented in MATLAB R2019a.
The input matrix $V_{\rm in}$ was set as a random matrix generated by the Mersenne Twister in MATLAB and each linear system was solved using the MATLAB command ``$\backslash$''.
\subsection{Experiment I: filter function}
Here, we compare the filter functions $f(\sigma)$ and $f_g(\sigma)$ for two cases.
Figure~\ref{fig:filter} shows that the values of the filter functions $|f(\sigma)|$ and $|f_g(\sigma)|$ for $[a,b] = [0.8,1.2]$ and $[10^{-3}, 10^{-1}]$.
\begin{figure}[t]
  \centering
  \begin{tabular}{cc}
    \begin{minipage}[t]{0.45\hsize}
      \centering
      \includegraphics[bb = 0 0 360 216, scale=0.45]{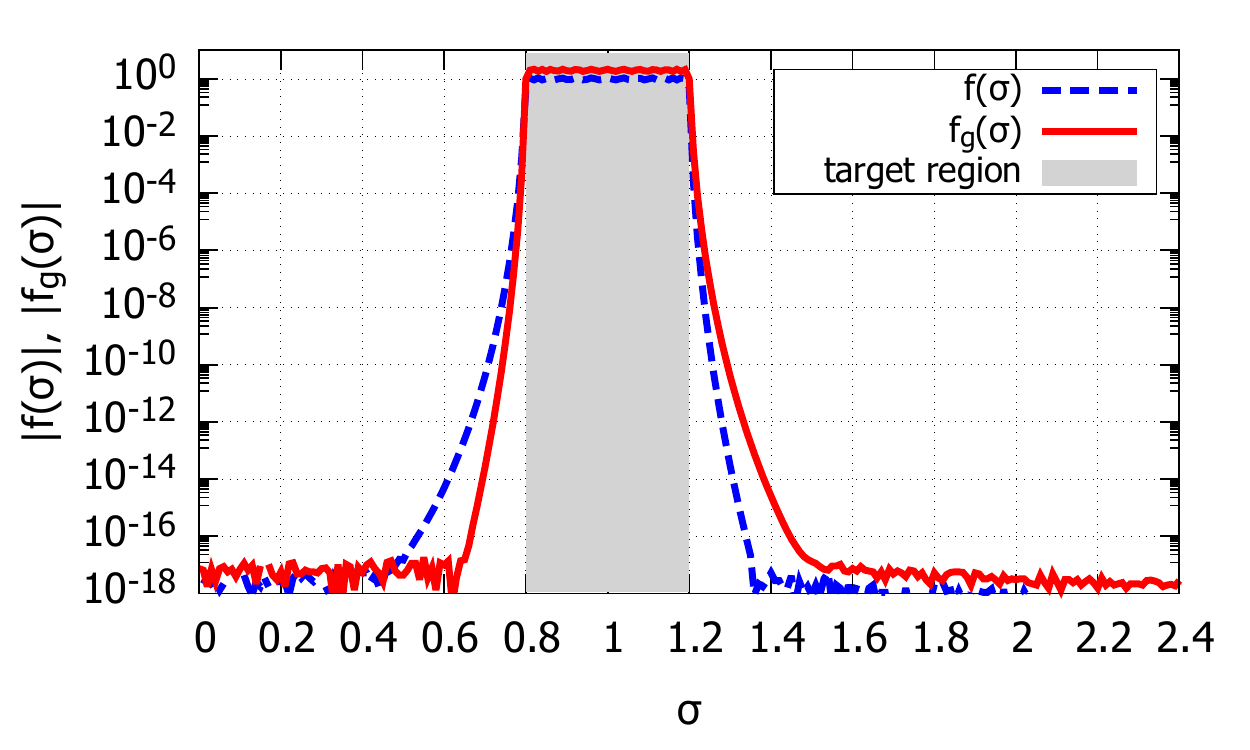} \\
      (a) $[a,b] = [0.8,1.2]$.
    \end{minipage}
    \begin{minipage}[t]{0.45\hsize}
      \centering
      \includegraphics[bb = 0 0 360 216, scale=0.45]{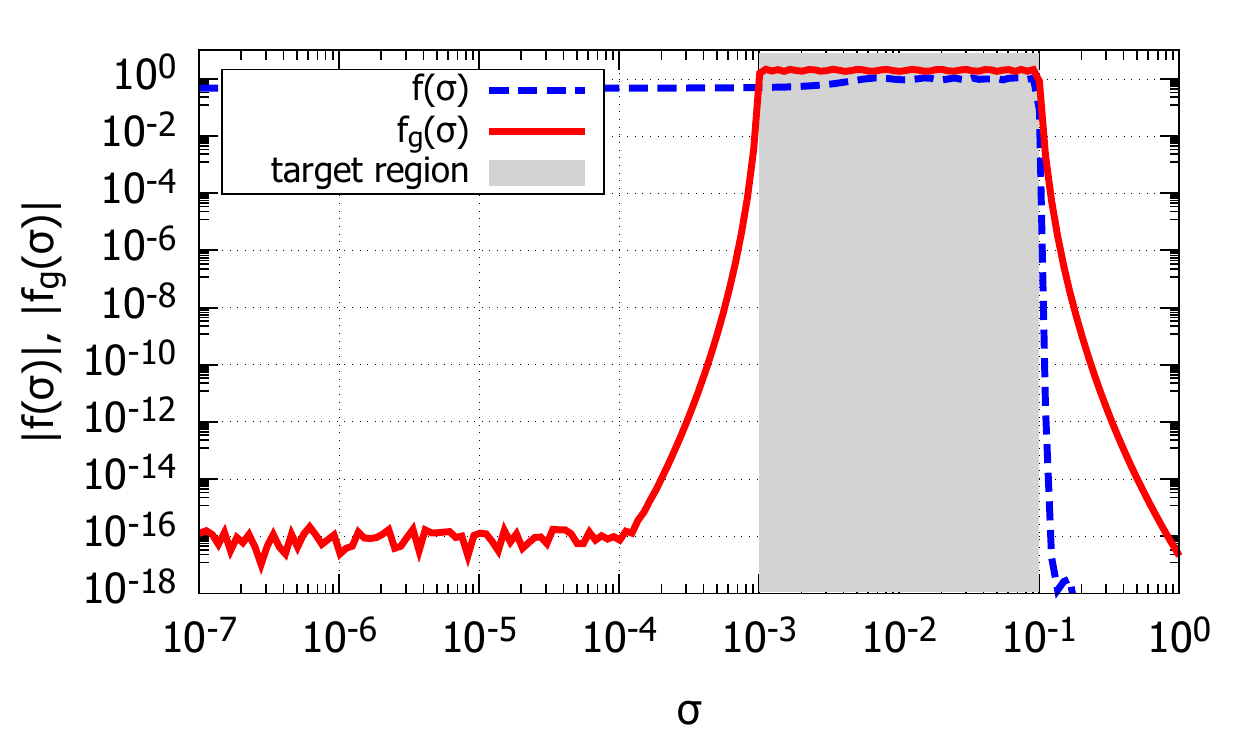} \\
      (b) $[a,b] = [10^{-3},10^{-1}]$.
    \end{minipage}
  \end{tabular}
  \caption{The values of $|f(\sigma)|$ and $|f_g(\sigma)|$.}
  \label{fig:filter}
\end{figure}
%
%
\par
As shown in Figure~\ref{fig:filter}(a), in the case that the target region is around 1, there is no large difference between the filter functions $f(\sigma)$ and $f_g(\sigma)$.
On the other hand, as shown in Figure~\ref{fig:filter}(b), in the case that the target region is cross to 0, there is large difference.
For the region of $\sigma > b$, both $|f(\sigma)|$ and $|f_g(\sigma)|$ decrease with increasing $\sigma$, specifically $|f(\sigma)|$ shows rapid decreasing.
For the region of $0 < \sigma < a$, $|f(\sigma)|$ shows large value $|f(\sigma)| \approx 0.5$.
Instead, $|f_g(\sigma)|$ drastically decrease.
This result indicates that, a cluster of singular value in the region $[0,10^{-4}]$ causes a negative effect on the accuracy of the SS-SVD based on $f(\sigma)$, but not the accuracy of the SS-SVD with nonlinear transform based on $f_g(\sigma)$.
\subsection{Experiment II: model problem}
In this subsection, we compare the accuracy of four methods:
\begin{itemize}
  \item Naive: The block SS--CAA method \cite{imakura2017block} via the symmetric eigenvalue problem \eqref{eq:sep};
  \item Naive with NT: Naive with the nonlinear trasnform;
  \item SS--SVD: The proposed method (Algorithm~\ref{alg:ss-svd});
  \item SS--SVD with NT: The proposed method with the nonlinear trainsform (Algorithm~\ref{alg:ss-svd-nonlinear}),
\end{itemize}
using two small model problems of size $m=1000, n = 200$.
For the model problem~1, we set
\begin{equation}
  \Sigma = {\rm diag}(0.005, 0.015, \ldots, 1.995) \in \mathbb{R}^{200 \times 200}
  \label{eq:model1}
\end{equation}
such that singular values are uniformly distributed, and for the model problem~2, we set
\begin{equation}
  \Sigma = {\rm diag}(10^{-10.0}, 10^{-9.95}, \ldots, 10^{-0.05}) \in \mathbb{R}^{200 \times 200}
  \label{eq:model2}
\end{equation}
such that singular values are uniformly distributed on the logarithmic scale.
Then, the matrix $A$ is constructed as $A = U \Sigma V^{\rm T}$ with orthogonal matrices $U$ and $V$ constructed from random matrices.
\begin{figure}[t]
  \centering
  \includegraphics[bb = 0 0 360 216, scale=0.45]{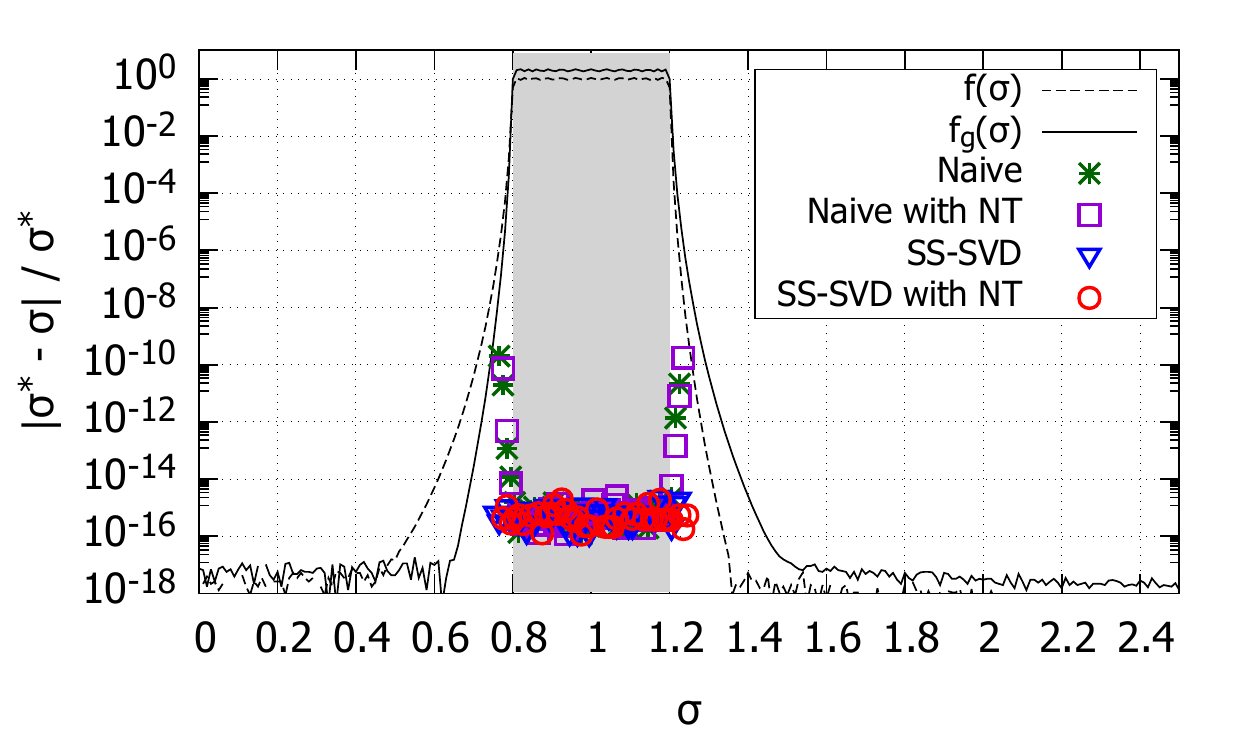} \quad
  \includegraphics[bb = 0 0 360 216, scale=0.45]{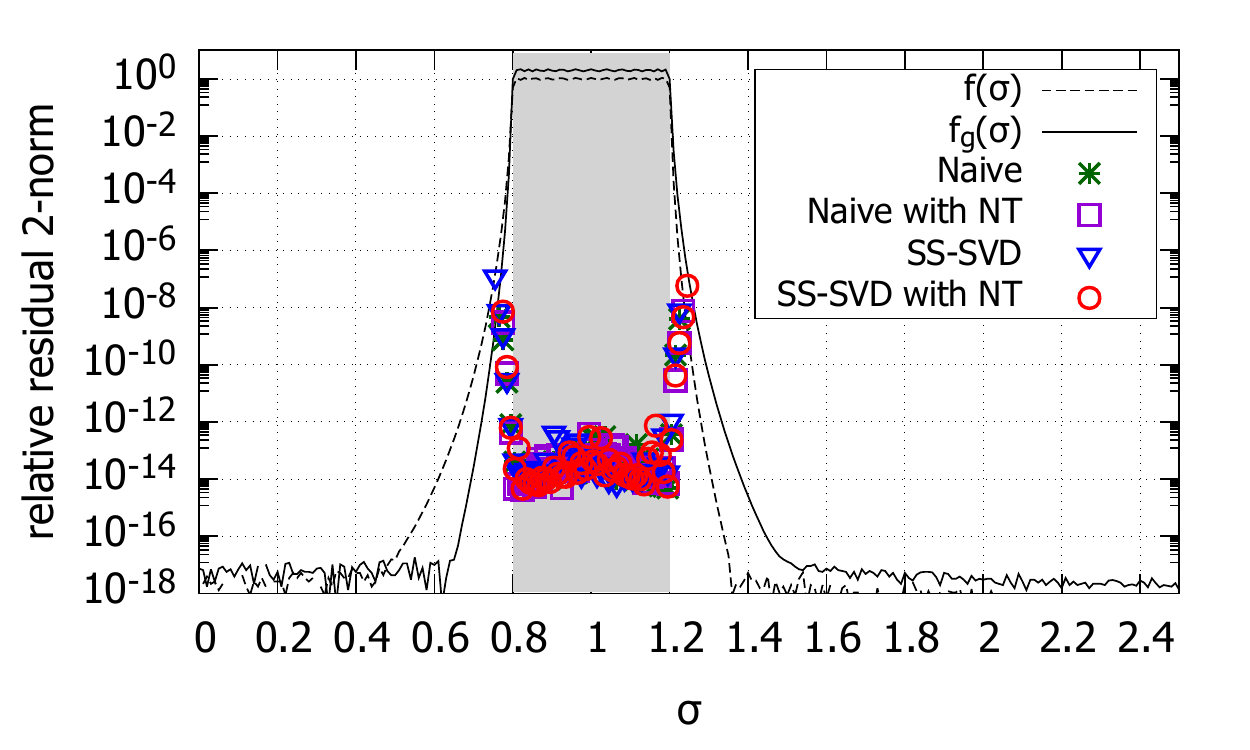} \\
  (a) The model problem 1 \eqref{eq:model1}. \\
  \includegraphics[bb = 0 0 360 216, scale=0.45]{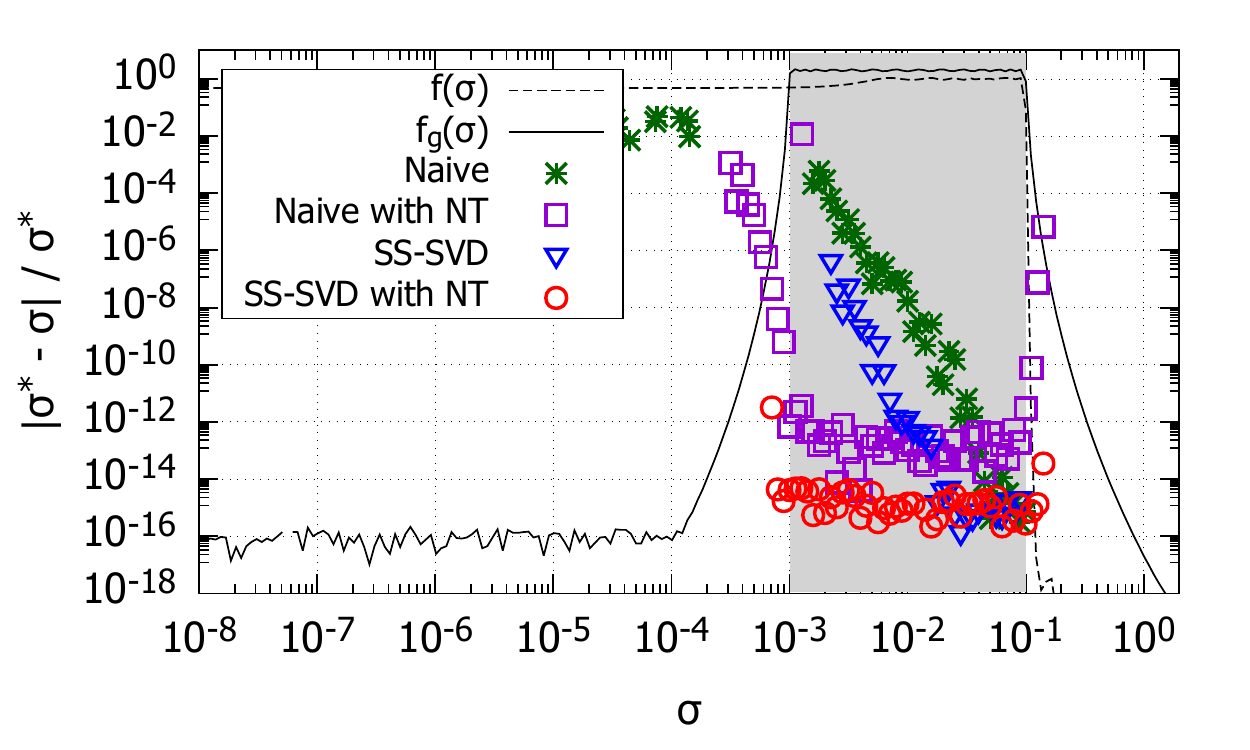} \quad
  \includegraphics[bb = 0 0 360 216, scale=0.45]{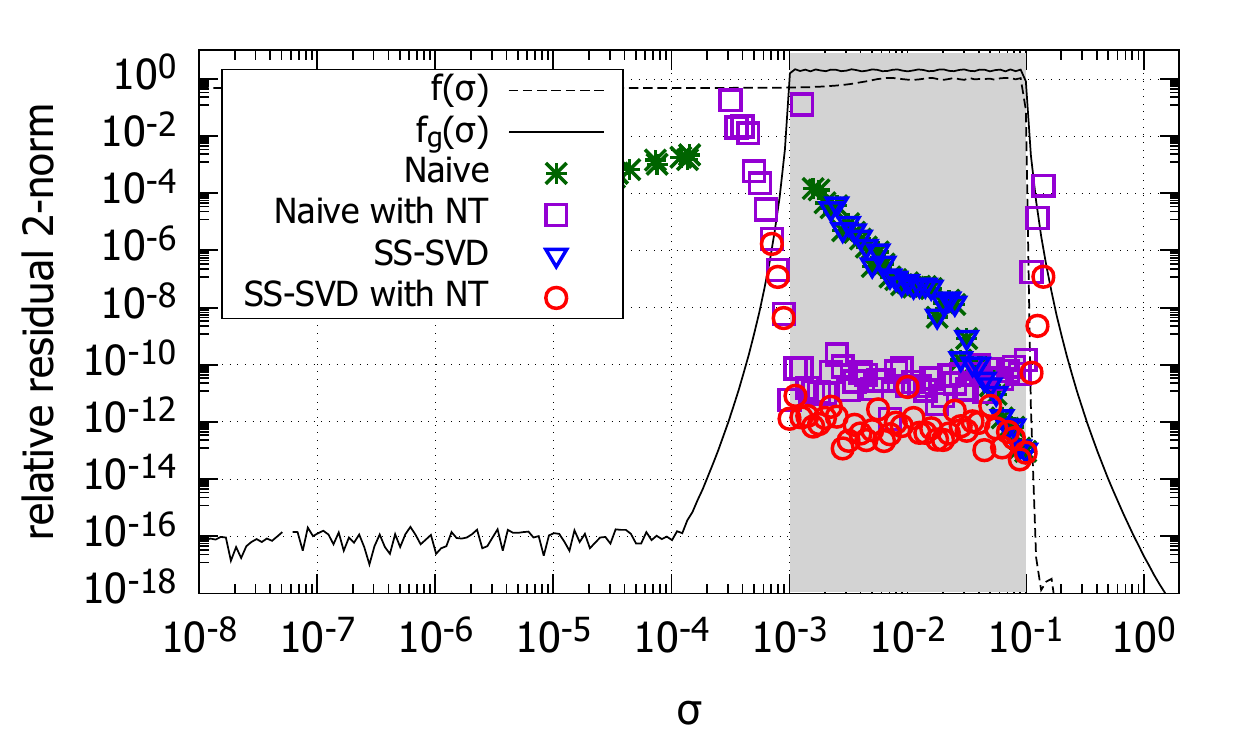} \\
  (b) The model problem 2 \eqref{eq:model2}.
  \caption{Relative error $|\sigma_i^\ast - \sigma_i | / \sigma_i^\ast$ of singular values and residual 2-norm $\| {\bm r}_i \|_2$.}
  \label{fig:ex2}
\end{figure}
\begin{figure}[t]
  \centering
  \begin{tabular}{cc}
    \begin{minipage}[t]{0.45\hsize}
      \centering
      \includegraphics[bb = 0 0 360 216, scale=0.45]{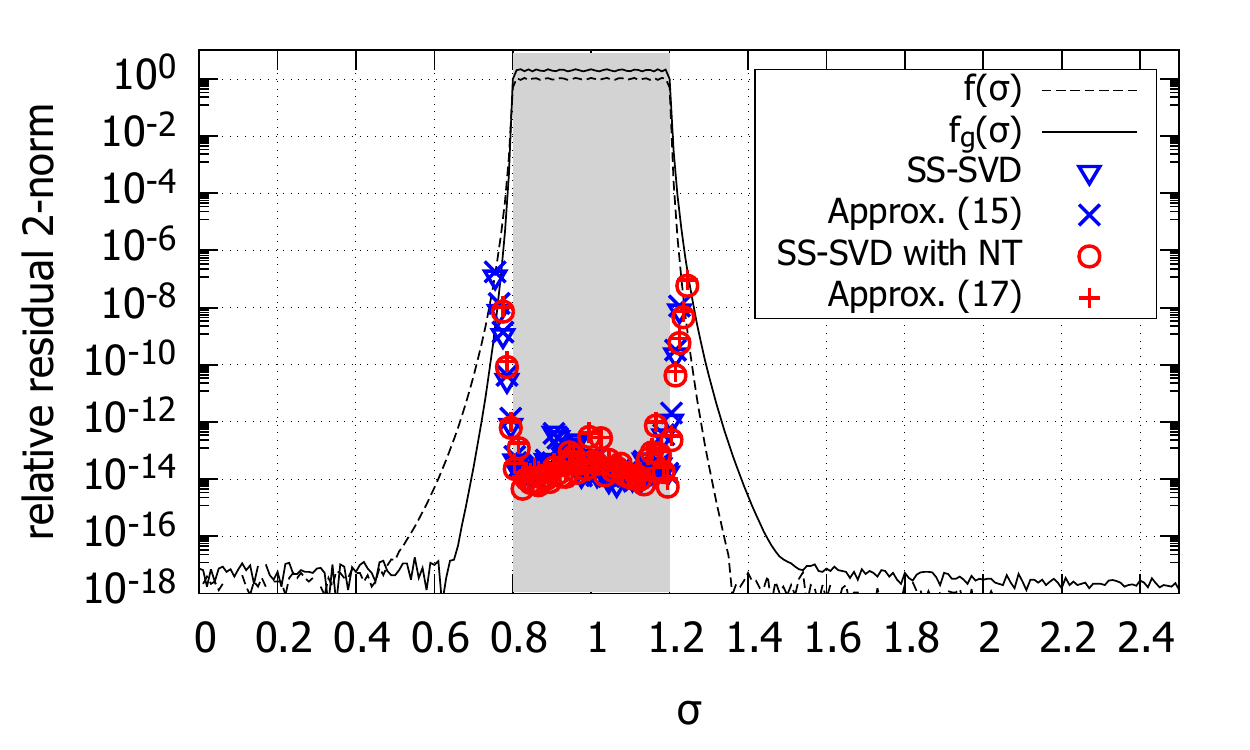} \\
      (a) The model problem 1 \eqref{eq:model1}.
    \end{minipage}
    \begin{minipage}[t]{0.45\hsize}
      \centering
      \includegraphics[bb = 0 0 360 216, scale=0.45]{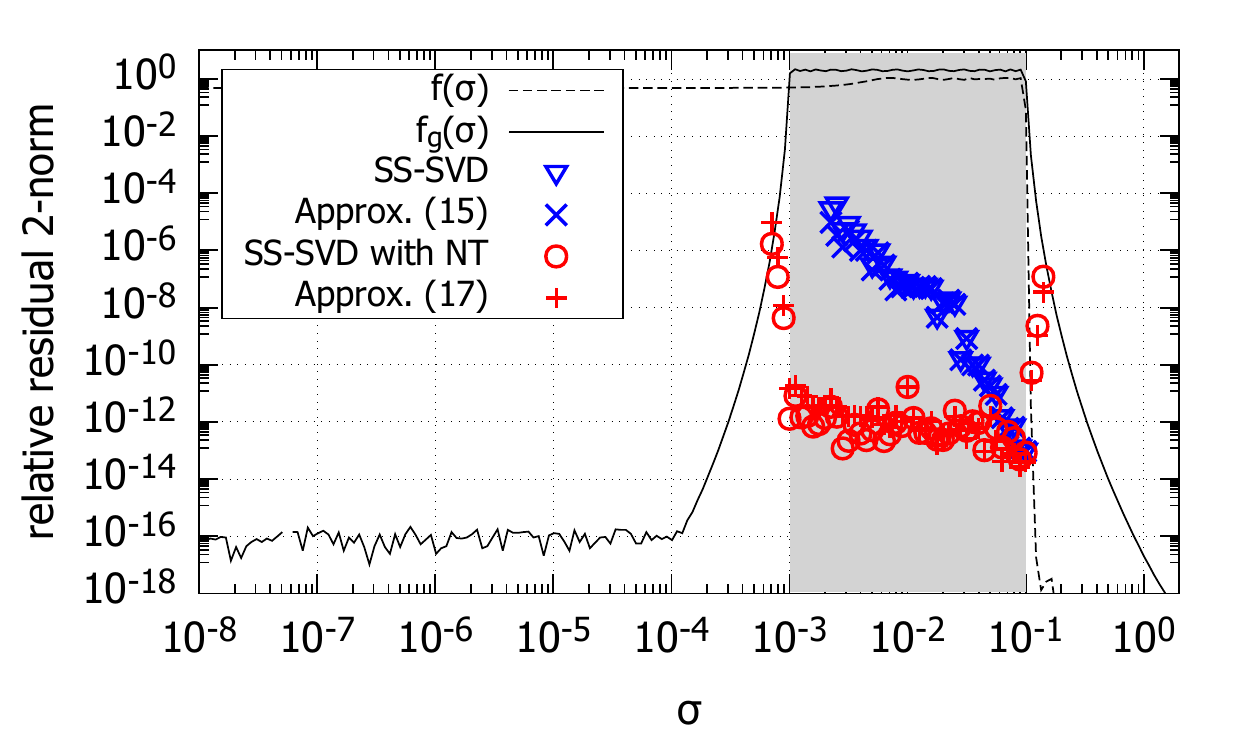} \\
      (b) The model problem 2 \eqref{eq:model2}.
    \end{minipage}
  \end{tabular}
  \caption{Residual norm and its approximation.}
  \label{fig:ex2_est}
\end{figure}
%
%
\par
For the model problem 1, we compute 40 singular triplets such that $\sigma_i \in [0.8,1.2]$ and for the model problem 2, we compute 40 singular triplets such that $\sigma_i \in [10^{-3},10^{-1}]$.
We set the parameters as $(L,M,N,\ell,\delta) = (20,4,32,1$, $10^{-20})$.
\par
In Figure~\ref{fig:ex2}, we show relative error of singular value $| \sigma_i^\ast - \widehat\sigma_i | / \sigma_i^\ast$, where $\sigma_i^\ast$ is the exact singular value, and residual 2-norm $\| {\bm r}_i \|_2 = \| A^{\rm T} \widehat{\bm u}_i - \widehat\sigma_i \widehat{\bm v}_i \|_2$.
In the case of the model problem 1 \eqref{eq:model1} whose singular values are uniformly distributed, as shown in Figure~\ref{fig:ex2}(a), all methods show almost the same accuracy both for the relative error and residual 2-norm.
On the other hand, in the case of the model problem 2 \eqref{eq:model2} whose singular values are uniformly distributed on the logarithmic scale, as shown in Figure~\ref{fig:ex2}(b), the nonlinear transformation shows drastically improves the accuracy.
Specifically, the SS-SVD with the nonlinear transformation shows the best results both for relative error and residual 2-norm.
\par
In Figure~\ref{fig:ex2_est}, we show the residual 2-norm and its estimations \eqref{eq:residual} and \eqref{eq:res_nonlinear} for the proposed methods.
We observe that the both approximations well estimated the exact residual 2-norms for both problems.
\subsection{Experiment III: real-world problem}
We evaluate the computation time of the proposed method compared with that of MATLAB functions \verb+svd+ for computing all the singular triplets and \verb+svds+ for computing partial singular triplets.
We used a $60000 \times 784$ sparse matrix obtained from a 10-class classification of handwritten digits (MNIST) \cite{lecun1998mnist}.
The average nonzero elements per row is 149.9.
We normalized the matrix as the largest singular value is 1.
For \verb+svds+, we computed the largest $10, 20, \dots, 160$ singular triplets.
For the proposed method, the target interval, the number $t$ of the target singular triplets and parameters $L, M$ are shown in Table~\ref{table:ex3}.
Here, we fix $M=4$ and set $L$ as $LM \approx 3t$.
We also set $(N,\ell,\delta) = (32, 1, 10^{-20})$.
\begin{table}[t]
  \caption{The target region $[a,b]$, the number $t$ of the target singular triplets and parameter setting for the proposed method.}
  \label{table:ex3}
\begin{center}
\begin{tabular}{rrrrrcrrrrr} 
\toprule
\multicolumn{4}{c}{Exterior} & & \multicolumn{4}{c}{Interior} \\ \cmidrule{1-4} \cmidrule{6-9}
\multicolumn{1}{c}{$[a,b]$} & \multicolumn{1}{c}{$t$} & \multicolumn{1}{c}{$L$} & \multicolumn{1}{c}{$M$} &  & \multicolumn{1}{c}{$[a,b]$} & \multicolumn{1}{c}{$t$} & \multicolumn{1}{c}{$L$} & \multicolumn{1}{c}{$M$} \\ \cmidrule{1-2} \cmidrule{3-4} \cmidrule{6-7} \cmidrule{8-9}
$[0.120, 1.01]$ &  21 & 15 & 4 &  & $[0.060, 0.08]$ &  17 & 15 & 4 \\
$[0.080, 1.01]$ &  40 & 30 & 4 &  & $[0.045, 0.08]$ &  39 & 30 & 4 \\
$[0.045, 1.01]$ &  79 & 60 & 4 &  & $[0.030, 0.08]$ &  85 & 60 & 4 \\
$[0.025, 1.01]$ & 153 &120 & 4 &  & $[0.020, 0.08]$ & 158 &120 & 4 \\
\bottomrule
\end{tabular}
\end{center}
\end{table}
\begin{figure}[t]
  \centering
  \includegraphics[bb = 0 0 360 216, scale=0.45]{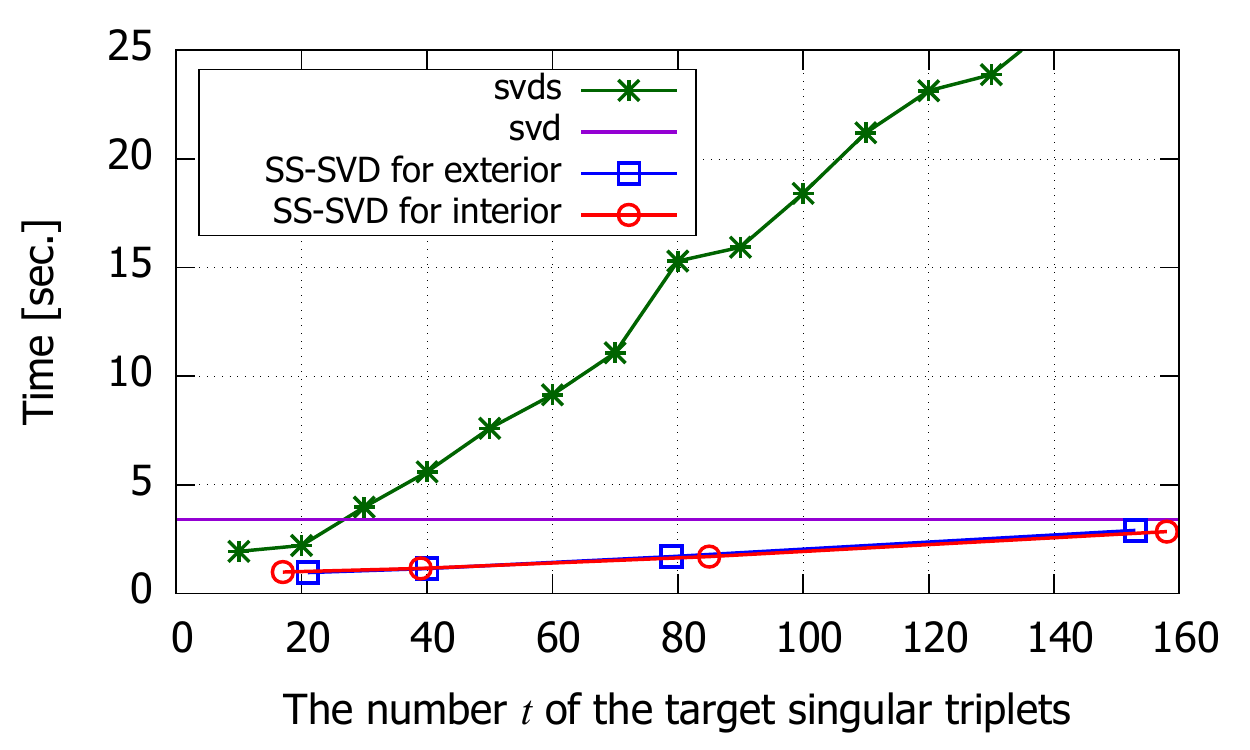}
  \caption{The computation time v.s. the number of the target singular triplets.}
  \label{fig:ex3}
\end{figure}

\begin{table}[t]
\caption{Breakdown of the computation time and accuracy of the proposed method.}
\label{table:ex3_result}
\begin{center}
\begin{tabular}{cccccccccc} 
\multicolumn{10}{c}{Exterior} \\
\toprule
\multicolumn{1}{c}{region} & & \multicolumn{5}{c}{Computation time [sec.]} &  & \multicolumn{2}{c}{accuracy} \\ \cmidrule{3-7} \cmidrule{9-10}
\multicolumn{1}{c}{$[a,b]$} & & \multicolumn{1}{c}{Steps 1--2} & \multicolumn{1}{c}{3} & \multicolumn{1}{c}{4} & \multicolumn{1}{c}{5} & \multicolumn{1}{c}{Total} &  & \multicolumn{1}{c}{error} & \multicolumn{1}{c}{residual} \\ \cmidrule{1-1} \cmidrule{3-7} \cmidrule{9-10}
$[0.120, 1.01]$ &  & 0.809 & 0.002  & 0.144  & 0.006  & 0.962  &  & 1.67E-15 & 1.57E-13 \\
$[0.080, 1.01]$ &  & 0.842 & 0.005  & 0.299  & 0.017  & 1.146  &  & 1.70E-15 & 1.90E-14 \\
$[0.045, 1.01]$ &  & 0.886 & 0.015  & 0.792  & 0.049  & 1.693  &  & 2.94E-15 & 1.03E-14 \\
$[0.025, 1.01]$ &  & 0.997 & 0.054  & 1.847  & 0.178  & 2.898  &  & 2.48E-15 & 3.90E-15 \\
\bottomrule
\\
\multicolumn{10}{c}{Interior} \\
\toprule
\multicolumn{1}{c}{region} & & \multicolumn{5}{c}{Computation time [sec.]} &  & \multicolumn{2}{c}{accuracy} \\ \cmidrule{3-7} \cmidrule{9-10}
\multicolumn{1}{c}{$[a,b]$} & & \multicolumn{1}{c}{Steps 1--2} & \multicolumn{1}{c}{3} & \multicolumn{1}{c}{4} & \multicolumn{1}{c}{5} & \multicolumn{1}{c}{Total} &  & \multicolumn{1}{c}{error} & \multicolumn{1}{c}{residual} \\ \cmidrule{1-1} \cmidrule{3-7} \cmidrule{9-10}
$[0.060, 0.08]$ &  & 0.830 & 0.002  & 0.150  & 0.006  & 0.983 &  & 1.09E-15 & 8.26E-14 \\
$[0.045, 0.08]$ &  & 0.858 & 0.005  & 0.286  & 0.017  & 1.148 &  & 2.62E-15 & 8.99E-14 \\
$[0.030, 0.08]$ &  & 0.886 & 0.015  & 0.795  & 0.047  & 1.695 &  & 2.18E-15 & 5.02E-13 \\
$[0.020, 0.08]$ &  & 1.024 & 0.056  & 1.762  & 0.184  & 2.842 &  & 2.48E-15 & 3.99E-16 \\
\bottomrule
\end{tabular}
\end{center}
\end{table}
\par
The computation time of all methods are shown in Figure~\ref{fig:ex3}.
The breakdown of the computation time and accuracy of the proposed method are also shown in Table~\ref{table:ex3_result}.
As shown in Figure~\ref{fig:ex3}, the computation time of \verb+svds+ increases with increasing the number $t$ of the target singular values.
Then, in this experiment, \verb+svd+ is faster than \verb+svds+ when $t > 30$.
Instead, the proposed method is much faster than \verb+svds+ and \verb+svd+ for both exterior and interior cases, although the computation time, specifically Step~4, of the proposed method increases with increasing $t$ and $L$.
\section{Conclusions}
\label{sec:conclusion}
Based on the concept of the complex moment-based eigensolvers, in this paper, we proposed a novel complex moment-based method to compute interior singular triplets \eqref{eq:psvd}.
We also analysed error bounds of the proposed method and proposed an improvement technique using a nonlinear transformation to improve accuracy of the proposed method.
The proposed method has high parallel efficiency as the complex moment-based parallel eigensolvers.
From the numerical experiments, the proposed complex moment-based method with the nonlinear transformation can compute accurate singular triplets for both exterior and interior problems.
The computation time of the proposed method is much faster than \verb+svds+ and \verb+svd+.
\par
In the future, we will evaluate the parallel performance of the proposed method for more large real-world problems.
\bibliography{mybibfile}
\bibliographystyle{elsart-num-sort}
\end{document}